\documentclass[11pt]{article}
\usepackage[margin=1in]{geometry}
\pdfoutput=1

\usepackage{amsmath}
\usepackage{amsfonts}
\usepackage{amsthm}
\usepackage{graphicx}
\usepackage{color}
\usepackage[usenames,dvipsnames]{xcolor}
\usepackage{hyperref}
\hypersetup{colorlinks=true,
    linkcolor=BrickRed,
    citecolor=OliveGreen,
urlcolor=MidnightBlue}
\usepackage{booktabs}
\usepackage{mathdots}

\newcommand{\R}{\mathbb{R}}
\newcommand{\C}{\mathbb{C}}

\renewcommand{\H}{\mathbb{H}}

\newcommand{\Sym}{\mathcal{S}}

\newcommand{\psd}{\succeq}
\newcommand{\pd}{\succ}

\newcommand{\Amap}{\mathcal{A}}

\DeclareMathOperator*{\conv}{conv}

\newcommand{\probclass}{\text{trigonometric Wahba problems}}
\newcommand{\probinstance}{\text{trigonometric Wahba problem}}

\newcommand{\moment}{\mathcal{M}}
\newcommand{\BlkHank}{\textup{Hankel}}
\newcommand{\BlkToep}{\textup{Toeplitz}}

\newcommand{\tr}{\textup{tr}}

\theoremstyle{plain}
\newtheorem{thm}{Theorem}[section]
\newtheorem{proposition}[thm]{Proposition}
\newtheorem{lemma}[thm]{Lemma}

\theoremstyle{definition}

\theoremstyle{remark}

\usepackage{subfigure}

\title{A convex solution to Psiaki's first joint attitude and spin-rate 
    estimation problem}
\author{James Saunderson \and Pablo A.~Parrilo \and Alan S.~Willsky
    \thanks{The authors are with the Laboratory for Information and 
    Decision Systems, Department of Electrical Engineering and Computer
    Science, Massachusetts Institute of Technology, Cambridge MA, 02139, USA\@.
    Email: \{jamess, parrilo, willsky\}@mit.edu. This research 
    was funded by the Air Force Office of Scientific Research
    under grants \#FA9550-12-1-0287 and \#FA9550-11-1-0305. 
    A prelimiary version of 
this work appears in~\cite{saunderson2014semidefiniterelaxations}.}}

\begin{document}
\maketitle

\begin{abstract}
    We consider the problem of jointly estimating the attitude 
    and spin-rate of a spinning spacecraft. 
    Psiaki (J. Astronautical Sci., 
    57(1-2):73--92, 2009) has formulated a family of 
    optimization problems that generalize the classical least-squares
    attitude estimation problem, known as Wahba's problem,
    to the case of a spinning spacecraft. 
    If the rotation axis is fixed and known, but the spin-rate
    is unknown (such as for nutation-damped spin-stabilized spacecraft)
    we show that Psiaki's problem can be reformulated
    exactly as a type of tractable convex optimization problem
   called a semidefinite optimization problem.
    This reformulation allows us to globally solve the problem 
    using standard numerical routines for semidefinite optimization. It
   also provides a natural semidefinite relaxation-based approach to more 
    complicated variations on the problem.
\end{abstract}


\section{Introduction}
\label{sec:intro}

Spacecraft attitude estimation is a fundamental problem, arising, for instance,
as a natural subproblem whenever attitude control is required. Since spacecraft
dynamics are non-linear, a typical and successful approach to attitude
estimation is to employ variants of the Extended Kalman Filter
(EKF)~\cite{lefferts1982kalman}.  As with any method based on linearization of
non-linear dynamics, EKF-based approaches can fail to converge given poor
initial estimates, and can become unstable in the presence of large
disturbances~\cite{psiaki2000attitude}.  Many truly non-linear attitude
estimation methods have also been proposed (see~\cite{crassidis2007survey} for
a survey). An important example is the static least-squares attitude estimation
problem known as Wahba's problem~\cite{wahba1965least}.  In Wahba's problem we
are simultaneously given a batch of vector measurements (from sun sensors, star
trackers, etc.) in the body frame and corresponding reference directions in an
inertial frame. The aim is to find the rotation matrix (i.e.\ direction cosine
matrix) that minimizes the sum of the squared errors between the transformed
reference directions and the observed vector measurements.  Wahba's problem, as
stated, applies most naturally to a static spacecraft.  Nevertheless, it has
also found use as a subroutine in various recursive estimation algorithms
including those that estimate the full dynamical state of the spacecraft (see,
e.g., \cite{psiaki2000attitude,gebre2000gyro}).

Recently Psiaki has posed a number of generalizations of Wahba's problem to the
case of a spinning spacecraft~\cite{psiaki2009generalized}. These problems aim
to simultaneously estimate the initial attitude and spin-rate (or, more
generally, initial angular momentum) of the spacecraft from vector
measurements, without the need for gyroscope measurements. These
generalizations are particularly suited to spin-stabilized spacecraft without
gyroscopes.  We describe Wahba's problem and Psiaki's generalizations formally
in Section~\ref{sec:jas}.

In this paper we focus on the simplest of Psiaki's generalizations of Wahba's
problem. We refer to this problem as \emph{Psiaki's first problem}.  In this
problem we assume the spacecraft is spinning at a constant unknown angular
velocity around a known (stable) inertia axis.  This setting is relevant for
nutation-damped spin-stabilized spacecraft~\cite{psiaki2009generalized}. The
aim is to estimate the initial attitude and the unknown spin-rate given a
sequence of noisy vector measurements obtained at certain sampling instants,
together with corresponding reference directions. Wahba's problem arises as the
special case where the spin-rate is zero.

\subsection{Main contribution}
\label{sec:contribtions}
Our main contribution is to show that, when the sampling period is constant,
Psiaki's first problem can be reformulated exactly as a semidefinite
optimization problem (see Theorem~\ref{thm:mainsdp}). Semidefinite optimization problems (described in
Section~\ref{sec:semidefinite}) are a family of convex optimization problems
that generalize linear programming and can be solved globally with provable
efficiency guarantees using standard software.  Reformulating Psiaki's first
problem as a semidefinite optimization problem means that it, like Wahba's
problem, can be solved efficiently and globally, to high precision, using
numerical methods.

A description of Psiaki's first problem as the solution to a semidefinite
optimization problem allows us to do more than just solve the original problem
as stated. It also allows us to take a semidefinite relaxation-based approach
to many variants on Psiaki's problem. We illustrate this in
Section~\ref{sec:variations} by considering the example of a version of
Psiaki's first problem where explicit bounds on the measurement errors are
incorporated into the formulation.

\subsection{Organization of the paper}
\label{sec:outline}

The remainder of the paper is organized as follows. In
Section~\ref{sec:notation} we summarize notation not defined elsewhere in the
paper. In Section~\ref{sec:jas} we first describe Psiaki's generalizations of Wahba's
problem for spinning spacecraft. We then show how to write Psiaki's first problem
as an instance of a family of problems we call \emph{\probclass{}} (see~\eqref{eq:genform1}).
We conclude the section with a summary of prior work on Psiaki's problems. 
In Section~\ref{sec:semidefinite} we briefly describe 
semidefinite optimization problems in general before presenting our semidefinite
optimization-based reformulation of \probclass{}, and in particular of 
Psiaki's first problem. We defer the proofs
to the Appendix.  In Section~\ref{sec:variations} we describe a
variant on Psiaki's first problem that incorporates additional bounds on the
measurement noise (if they are available) and show how to extend our
semidefinite optimization-based reformulation of Psiaki's first problem to a
semidefinite relaxation of this variant. We also describe the results of a
simple numerical experiment comparing Psiaki's first problem and this
variant. In Section~\ref{sec:directions} we discuss possible future
research related to the work in this paper.  

\subsection{Notation}
\label{sec:notation}
We briefly summarize notation used throughout the body of the paper. Additional 
notation that is used only in the Appendix is introduced separately there.

\paragraph{Spaces}
Denote by $\R^{n\times n}$ the space of $n\times n$ real matrices.  If $X\in
\R^{n\times n}$ let $X^T$ be its transpose.  Let $\Sym^n$ be the space of
$n\times n$ symmetric matrices (i.e.\ matrices for which $X = X^T$).  Let
$\Sym_+^n$ denote the set of $n\times n$ symmetric positive semidefinite
matrices (i.e.\ $X\in \Sym_+^n$ if and only if $u^TXu\geq 0$ for all $u\in
\R^n$).  If $X\in \Sym_+^n$ we write $X \psd 0$ when the dimension is clear
from the context.

\paragraph{Inner products}
If $x,y\in \R^n$ then $\langle x,y\rangle = \sum_{i=1}^{n}x_i y_i$. 
If $x\in \R^n$ then define $\|x\| = \langle x,x\rangle^{1/2} = 
\left(\sum_{i=1}^{n}x_i^2\right)^{1/2}$ to be the usual Euclidean norm.
If $X,Y\in \R^{n\times n}$ then define an inner product by  
$\langle X,Y\rangle = \tr(X^TY) = \sum_{i,j=1}^{n}X_{ij}Y_{ij}$. 

\paragraph{Convexity}
Given a subset $S\subset \R^n$ then 
\[ \conv(S) = \left\{\sum_{i}\lambda_i x_i:\; \sum_{i}\lambda_i = 1,
\;\;x_i\in S,\; \lambda_i \geq 0,\;\text{for all $i$}\right\}\]
is the set of all convex combinations of elements of $S$. From the point of
view of optimization, if $c\in \R^n$ and $S$ is compact then 
\[ \max_{x\in S} \langle c,x\rangle = \max_{x\in \conv(S)} \langle c,x\rangle\]
so the optimal cost is the same whether we optimize the linear functional
defined by $c$ over $S$ or over its convex hull~\cite[Theorem
32.2]{rockafellar1997convex}.

\paragraph{Block matrices}
        If $T_0,T_1,\ldots,T_{N}$ are $d\times d$ matrices with $T_0$ being symmetric, 
        define the corresponding $d(N+1)\times d(N+1)$
        symmetric block Toeplitz matrix by
        \begin{equation}
            \label{eq:blktoep}
            \BlkToep(T_0,T_1,\ldots,T_N) = 
            \begin{bmatrix} T_0 & T_1 & T_2 & \cdots & T_N\\
            T_1^T & T_0 & T_1 & \ddots & \vdots\\
            T_2^T & T_1^T & \ddots & \ddots & \vdots\\
            \vdots & \ddots & \ddots & \ddots & T_1\\
            T_N^T & \cdots & \cdots & T_1^T & T_0\end{bmatrix}.
    \end{equation}
    Similarly if $S_1,S_2,\ldots,S_{2N+1}$ are symmetric $d\times d$, matrices 
    define the corresponding $d(N+1)\times d(N+1)$ block Hankel matrix by 
    \begin{equation}
        \label{eq:blkhank}
        \BlkHank(S_1,S_2,\ldots,S_{2N+1}) = 
        \begin{bmatrix} S_1 & S_2 & \cdots & S_{N} & S_{N+1}\\
        S_2 &  &  & S_{N+1} & S_{N+2}\\
        \vdots &  & \iddots &  & \vdots\\
        S_{N} & S_{N+1} &  &  & S_{2N}\\
S_{N+1} & S_{N+2} & \cdots & S_{2N} & S_{2N+1}\end{bmatrix}.
    \end{equation}

\paragraph{Unit quaternion parameterization of rotations}
We make extensive use of the quadratic parameterization of $SO(3)$, the set of rotation (or direction-cosine) 
matrices,  by unit quaternions, denoted by $\H$. 
Throughout we think of the unit quaternions geometrically as the unit sphere in $\R^4$
i.e.\ $\H = \{q\in \R^4:\|q\|=1\}$. We only ever work with a unit quaternion $q\in\H$ via the 
positive semidefinite matrix $qq^T$, avoiding the sign
ambiguity that would arise if we were to try to work directly with variables
$q\in \H$. It is enough only to consider $qq^T$ because any element of
$SO(3)$ can be expressed as $\Amap(qq^T)$ where $q\in \H$ and
$\Amap:\Sym^4\rightarrow \R^{3\times 3}$ is the linear map 
defined (following the convention in~\cite{crassidis2007survey}) by
\begin{equation}
    \label{eq:amap}
\Amap(Z) := \begin{bmatrix} Z_{11} - Z_{22} - Z_{33} + Z_{44} & 2Z_{12} + 2Z_{34} & 
    2Z_{13} - 2Z_{24}\\2Z_{12} - 2Z_{34} & -Z_{11} + Z_{22} - Z_{33} + Z_{44} & 2Z_{23}+2Z_{14}\\
2Z_{13} + 2Z_{24} & 2Z_{23} - 2Z_{14} & -Z_{11} - Z_{22} + Z_{33} + Z_{44}\end{bmatrix}.
\end{equation}
The adjoint of $\Amap$ (with respect to the inner product on matrices) is
$\Amap^*:\R^{3\times 3}\rightarrow \Sym^4$ defined by
\begin{equation}
    \label{eq:amapstar}
\Amap^*(Y) := \begin{bmatrix} Y_{11}-Y_{22}-Y_{33} & Y_{12} + Y_{21} & Y_{13}+Y_{31} & Y_{23}-Y_{32}\\
Y_{12}+Y_{21} & -Y_{11}+Y_{22} - Y_{33} & Y_{23}+Y_{32}& -Y_{13} + Y_{31}\\
Y_{13}+Y_{31} & Y_{23} + Y_{32} & -Y_{11} - Y_{22} + Y_{33} & Y_{12}-Y_{21}\\
Y_{23} - Y_{32} & -Y_{13}+Y_{31} & Y_{12}-Y_{21} & Y_{11} + Y_{22} + Y_{33}\end{bmatrix}.
\end{equation}
In other words for any $Z\in \Sym^4$ and any $Y\in \R^{3\times 3}$, we have the identity
\begin{equation}
    \label{eq:adjoint}
    \langle \Amap(Z),Y\rangle = \langle Z,\Amap^*(Y)\rangle.
\end{equation}

\section{Psiaki's generalizations of Wahba's problem for spinning
    spacecraft}
\label{sec:jas}

In this section we describe Wahba's problem~\cite{wahba1965least} and Psiaki's
generalizations to the case of a spinning
spacecraft~\cite{psiaki2009generalized}. For reasons discussed in
Section~\ref{sec:psiaki-form} we subsequently focus on the simplest of Psiaki's
problems: jointly estimating the attitude and spin-rate of a spacecraft
spinning around a stable inertia axis at a constant unknown rate.  In this case
we show how to reformulate the resulting optimization problem in the general
form
\begin{equation}
    \label{eq:genform1}\max_{\substack{Q\in SO(3)\\\omega\in [-\pi,\pi)}}
    \langle A_0,Q\rangle +
    \sum_{n=1}^{N}\left[ \langle A_n,\cos(\omega n)Q\rangle + 
        \langle B_n,\sin(\omega n)Q\rangle\right] 
\end{equation}
for appropriate collections of $3\times 3$ matrices $(A_n)_{n=0}^{N}$ and
$(B_n)_{n=1}^{N}$. Throughout, we call problems in the form~\eqref{eq:genform1}
\emph{\probclass{}}. In Section~\ref{sec:semidefinite} to follow, we show how
to reformulate \probclass{} as semidefinite optimization problems. 

\subsection{Wahba's problem}
\label{sec:wahba}
We briefly describe Wahba's least squares attitude estimation problem
posed in~\cite{wahba1965least} with solutions published 
in~\cite{farrell1966least}.

\paragraph{Vector measurements}
Suppose we are given a batch of noisy unit vector measurements 
$y_0,y_1,\ldots,y_N$ in the body frame
(obtained from star trackers, sun sensors, 
magnetometers, etc.) of corresponding unit reference directions 
$x_0,x_1,\ldots,x_N$ in the inertial frame.

\paragraph{Least squares objective}
Wahba's problem is to find the rotation matrix $Q\in SO(3)$
that transforms the reference directions to best fit the measured
vector measurements in the weighted least squares sense
by solving
\begin{equation}
\label{eq:wahba0}
\min_{Q\in SO(3)}\; \sum_{n=0}^{N}\frac{\kappa_n}{2}\|y_n - Qx_n\|^2
\end{equation}
where $\kappa_0,\kappa_1,\ldots,\kappa_N$ are non-negative scalar weights
that one would take to be larger for measurements with smaller noise 
variance. Since 
$\|Qx\|^2 = \|x\|^2$
for all $x\in \R^3$ we can expand the squares and see that this optimization
problem is equivalent to
\begin{equation}
    \label{eq:wahba}
    \max_{Q\in SO(3)}\; \langle \sum_{n=0}^{N}\kappa_n y_nx_n^T,Q\rangle
\end{equation}
where we have dropped an additive constant of 
$\sum_{n=0}^{N}\frac{\kappa_n}{2}(\|y_n\|^2+\|x_n\|^2)$.

\subsection{Psiaki's generalizations}
\label{sec:psiaki-form}
We now describe Psiaki's generalizations of Wahba's problem, and show
how Wahba's problem arises as a special case. 

\paragraph{Rigid body (Euler) equations}
Let $Q(t_0)\in SO(3)$ denote the initial attitude of the 
spacecraft, $\Omega(t_0)\in \R^3$ the initial body angular 
velocity, and $I_1\geq I_2\geq I_3$ the principal moments of 
inertia. Assuming the spacecraft undergoes torque-free motion about
its centre of mass then for $t\geq t_0$ the attitude $Q(t)$ and the body angular
velocity $\Omega(t) := \begin{bmatrix} \omega_1(t) & \omega_2(t) &
    \omega_3(t)\end{bmatrix}^T$ satisfy the rigid body equations:
\begin{equation}
    \label{eq:ode}
    \begin{matrix}
        I_1\;\dot{\omega}_1(t) = (I_2-I_3)\,\omega_2(t)\,\omega_3(t)\\
    I_2\;\dot{\omega}_2(t) = (I_3-I_1)\,\omega_3(t)\,\omega_1(t)\\
    I_3\;\dot{\omega}_3(t) = (I_1-I_2)\,\omega_1(t)\,\omega_2(t)\\
\end{matrix}\quad\text{and}\quad
    \dot{Q}(t) = \begin{bmatrix} 0 & -\omega_3(t) & \phantom{-}\omega_2(t)\\
        \phantom{-}\omega_3(t)
        & 0 & -\omega_1(t)\\-\omega_2(t) & \phantom{-}\omega_1(t) & 0\end{bmatrix}Q(t).
\end{equation}
Note that for every $t \geq t_0$ and every $\Omega(t_0)$ we have that $Q(t) =
\Phi(t-t_0;\Omega(t_0))Q(t_0)$ for some map $\Phi$ taking values in $SO(3)$. In
particular $Q(t)$ is always linear in the initial attitude $Q(t_0)$. 

\paragraph{Vector measurements}
Let $t_0,t_1,\ldots,t_N$ be a finite set of 
sampling instants. Assume, at sample instant $t_n$, that we are given a noisy unit vector 
measurement $y_n$ in the spacecraft body frame of a corresponding reference directions 
$x_n$ in the inertial frame. 

\paragraph{Least squares objective}

Following Wahba's least-squares-based objective, Psiaki suggests solving the
following weighted least-squares problem to estimate the initial attitude and
body angular velocity of the spacecraft, given only the vector measurements
$(y_n)_{n=0}^{N}$ and the reference directions $(x_n)_{n=0}^{N}$:
\begin{equation}
    \label{eq:psiaki-lsq}
    \min_{Q(t_0),\Omega(t_0)}\sum_{n=0}^{N}\frac{\kappa_n}{2}\|y_n -
    Q(t_n)x_n\|_2^2\end{equation}
subject to $Q(t)$ satisfying~\eqref{eq:ode} with initial conditions $Q(t_0)$
and $\Omega(t_0)$. Just as for Wahba's problem, the $\kappa_n$ are non-negative 
scalars.

\paragraph{Dependence on $\Omega(t_0)$}
In general, the dependence of $Q(t)$ on the initial body angular velocity
$\Omega(t_0)$ is quite complicated.  The relationship between $Q(t)$ and
$\Omega(t_0)$ simplifies under additional assumptions on $\Omega(t_0)$ and the
inertia tensor of the spacecraft. We now summarize these simplified problems
and name them for later reference.
\begin{description}
    \item[Wahba's problem] If $\Omega(t_0) = 0$, then $Q(t) = Q(t_0)$ for all
        $t\geq t_0$ and so the spacecraft is stationary. Adding this as a constraint
        we recover Wahba's original formulation~\eqref{eq:wahba0}.
    \item[Psiaki's first problem] Suppose $\Omega(t_0)$
            is aligned with the major inertia
        axis, and (without loss of generality) this is the first axis direction in body coordinates.
       Then $\Omega(t_0)= \begin{bmatrix}\omega & 0 & 0\end{bmatrix}^T$
      and so the dynamical constraints~\eqref{eq:ode} reduce to
        \begin{equation} \label{eq:spineq} Q(t) = \begin{bmatrix}1 & 0 & 0\\0 &
                \cos(\omega t) & -\sin(\omega t)\\0 & \sin(\omega t) &
                \cos(\omega t) \end{bmatrix}Q(t_0) \end{equation}
        where $\omega$ is the spin-rate (in rad/second). In this case the
        spacecraft is spinning with an \emph{unknown} constant angular velocity
        $\omega$ around a \emph{known} axis (fixed in body coordinates).  Minimizing 
        the least-squares objective~\eqref{eq:psiaki-lsq} subject to the constraints~\eqref{eq:spineq}
        is the first generalization of Wahba's problem posed
        in~\cite{psiaki2009generalized}, and is relevant for a nutation damped
        spin-stabilized spacecraft.
\item[Psiaki's second problem] If $\Omega(t_0)$ is unconstrained and no 
    additional assumptions are made about the moments of inertia
    of the spacecraft, we obtain the second generalization
    of Wahba's problem posed in~\cite{psiaki2009generalized}. In this setting
    the dependence of $Q(t)$ on $\Omega(t_0)$ is more complicated.
    This case is discussed further in~\cite{psiaki2012numerical} (see Section~\ref{sec:related} 
    to follow). 
\end{description}
In each case, Psiaki's formulations involve solving non-convex optimization
problems of the form in~\eqref{eq:psiaki-lsq} subject to dynamical constraints. 

\paragraph{Focus of the paper}
For the remainder of the  paper we focus on Psiaki's first problem,
because in this case $Q(t)$ only depends on the initial body angular 
velocity through $\cos(\omega t)$ and $\sin(\omega t)$.
In addition to focusing on Psiaki's first problem, we also assume that the
sampling instants $t_0,t_1,\ldots,t_N$ are equally spaced. As such we assume
there is some $\tau$ such that $t_n = n\tau$ for $n=0,1,\ldots,N$. 

This paper does not address Psiaki's more general second problem, where 
the dependence of $Q(t)$ on $\Omega(t_0)$ is significantly more complicated.
It would be very interesting if the techniques we develop can be extended 
to this more general situation. 

\paragraph{Aliasing}
Since we only observe $\omega$ via vector measurements at time instants that
are integer multiples of $\tau$, from the data alone we cannot distinguish
between spin rates at different integer multiples of $2\pi/\tau$ due to
aliasing. Hence we assume that $\omega\in [-\pi/\tau,\pi/\tau)$ so that
it is possible to determine the unknown spin-rate from the data. (We could,
alternatively, fix some $a$ rad/second and assume $\omega \in
[a,a+2\pi/\tau)$.) In a Bayesian formulation of the problem, we could interpret
this as encoding prior information on the spin rate.  

\paragraph{Reformulation} We now reformulate
Psiaki's first problem as a \probinstance{}.
Since $\|Q(t)x_n\|^2 = \|x_n\|^2$ for all $t$ and $n$, 
observe that with $t_n = n\tau$ the optimization 
problem~\eqref{eq:psiaki-lsq} can be rewritten as
\begin{align}
    \label{eq:inter}
    \min_{\substack{Q(0)\in SO(3)\\\omega\in [-\pi/\tau,\pi/\tau)}}&\; 
    \sum_{n=0}^{N}\frac{\kappa_n}{2}[\|y_n\|^2 
    - 2\langle y_n,Q(n\tau)x_n\rangle + \|x_n\|^2]\\
    \text{s.t.}\quad&
    Q(n\tau) = \begin{bmatrix} 1 & 0 & 0\\0 & \cos(n\tau\omega) & -\sin(n\tau\omega)\\
        0 & \sin(n\tau\omega) & \cos(n\tau\omega)\end{bmatrix}Q(0). 
    \label{eq:spin2}
\end{align}
Putting $\omega' = \tau\omega$, we see that this is equivalent, 
as an optimization problem, to
    \begin{equation}
        \label{eq:psiaki1}
        \max_{\substack{Q\in SO(3)\\\omega'\in [-\pi,\pi)}} \langle A_0,Q\rangle + \sum_{n=1}^{N}\left[\langle A_n,\cos(n\omega')Q\rangle
            + \langle B_n,\sin(n\omega')Q\rangle \right]
    \end{equation}
    where 
\begin{equation}
\label{eq:A0first}
A_0 = \kappa_0\,y_0 x_0^T + \begin{bmatrix} 1 & 0 & 0\\0 & 0 & 0\\0 & 0 & 0\end{bmatrix}\left(\sum_{n=1}^{N}\kappa_n\,y_n x_n^T\right)
\end{equation}
    and for $n=1,2,\ldots,N$,
\begin{equation}
\label{eq:AnBn}
A_n = \kappa_n\begin{bmatrix} 0 & 0 & 0\\0 & 1 & 0\\0 & 0 & 1\end{bmatrix}y_n x_n^T\quad\text{and}\quad
B_n = \kappa_n\begin{bmatrix} 0 & 0 & 0\\0 & 0 & 1\\0 & -1 & 0\end{bmatrix}y_n x_n^T.
\end{equation}
We have now expressed Psiaki's first problem in the general form described 
in~\eqref{eq:genform1}.

\subsection{Prior work and alternative solution methods for Psiaki's problems}
\label{sec:related}
In this section we summarize previous approaches to Psiaki's generalizations 
of Wahba's problem for spinning spacecraft. We then briefly discuss 
a simple discretization-based approach, 
implicit in the work of Psiaki and Hinks~\cite{psiaki2012numerical},
for solving Psiaki's problems globally.

Psiaki's original paper~\cite{psiaki2009generalized} describes a method to
globally solve Psiaki's first problem when two noise-free
vector measurements (sampled at distinct times) are used. In this situation
the problem reduces to finding all the solutions of the corresponding 
non-linear equations satisfied by the initial attitude and spin-rate.
This method seems quite sensitive to measurement noise, and is unable 
to exploit additional measurements to mitigate the effects of noise.
(The advantages of incorporating multiple measurements are demonstrated
in Section~\ref{sec:expts}.)

In subsequent work~\cite{hinks2011solution} Hinks and Psiaki describe an
approach to Psiaki's second problem under the assumption that the spacecraft is
axially symmetric and exactly three noise-free vector
measurements are used. In this case it is again possible to find an initial body
angular velocity $\Omega(t_0)$ and an initial attitude that are consistent with
the measurements  by solving a set of non-linear equations. They suggest different formulations of
these equations, and apply Newton's method (with possibly many different
initializations) to obtain a solution to the equations. Again this approach is
likely to be useful only when there is very little noise.

In later work~\cite{psiaki2012numerical} Psiaki and Hinks describe a method to
find local optima of Psiaki's first and second problems 
(with no additional assumptions)
by a novel alternating optimization scheme. The main idea is that for fixed
$\Omega(t_0)$, each point $Q(t_0),Q(t_1),\ldots,Q(t_N)$ on the trajectory is
linear in $Q(t_0)$. Hence if we can compute the trajectory 
$(Q(t_n))_{n=0}^{N}$ for fixed $\Omega(t_0)$ we can minimize the 
objective function of~\eqref{eq:psiaki-lsq} over $Q(t_0)$ for 
fixed $\Omega(t_0)$ by solving an instance of Wahba's problem. 
To obtain the trajectory $(Q(t_n))_{n=0}^{N}$ for fixed $\Omega(t_0)$, 
Psiaki and Hinks suggest numerically solving the rigid body equations. 
For the other part of the alternating optimization scheme, 
they employ a trust-region method to locally
optimize over $\Omega(t_0)$ for fixed $Q(t_0)$. 
As presented this problem only finds local optima for 
$\Omega(t_0)$ and $Q(t_0)$. Nevertheless this method makes very 
few assumptions, and can incorporate many measurements and so should 
behave well in the presence of measurement noise. 

A simpler, but much more naive, strategy would be to discretize
the space of $\Omega(t_0)$, solve (in parallel) the corresponding 
instance of Wahba's problem for each value of $\Omega(t_0)$,
then output the pair $(\Omega(t_0),Q(t_0))$ with the smallest cost.
This is a reasonable strategy for 
Psiaki's first problem since aliasing issues mean there is always 
an optimal $\omega$ in the interval $[-\pi/\tau,\pi/\tau)$.
A clear downside of this discretization approach when compared 
with the semidefinite 
optimization-based methods we describe in Section~\ref{sec:semidefinite}
is that it is expensive to obtain global solutions of high accuracy. 
Furthermore, the semidefinite optimization-based formulation easily 
extends to give semidefinite optimization-based formulations for more
general problems (see Section~\ref{sec:variations}) where the subproblems
for fixed $\Omega(t_0)$ do not reduce to instances of Wahba's problem.

\section{Semidefinite optimization reformulations}
\label{sec:semidefinite}

The main aim of this section is to describe how to reformulate \probclass{},
and hence Psiaki's first problem (which is a special case), as semidefinite
optimization problems. Before doing so, we briefly explain what semidefinite
optimization problems are, and what we mean by a semidefinite reformulation of
an optimization problem.  We illustrate this in Section~\ref{sec:wahba-sdp} by
giving a semidefinite reformulation of Wahba's problem that can be thought of
as a more flexible description of the $q$-method~\cite{keat1977analysis}.  In
Section~\ref{sec:psiaki-sdp} we give a semidefinite reformulation of
\probclass{}, before giving, in Section~\ref{sec:code}, pseudocode illustrating
how to implement the semidefinite optimization problems we formulate using
generic semidefinite optimization solvers.

\subsection{Semidefinite optimization}
\label{sec:sdp}
Semidefinite optimization problems are convex optimization problems
of the form
\[ \max_{x}\; \langle c,x\rangle \quad\text{s.t.}\quad
A_0 + \sum_{i=1}^{n}A_ix_i \psd 0\]
where $x\in \R^n$ is a vector of decision variables, $c\in \R^n$ represents a 
linear cost functional, the matrices $A_0,A_1,\ldots,A_n$ are symmetric 
$m\times m$ matrices. Recall that $X \psd 0$ means that the symmetric matrix 
$X$ is positive semidefinite. An expression of the form
\[ A(x) = A_0 + \sum_{i=1}^{n} A_ix_i \psd 0\]
is often called a \emph{linear matrix inequality} because it is linear in the 
decision variable $x$. 

Semidefinite optimization problems can be solved to any desired accuracy
in time polynomial in $n$ and $m$
using standard software based on interior point methods~\cite{vandenberghe1996semidefinite}. The semidefinite
optimization problems that arise in this paper have additional structure that
could be exploited to obtain even more efficient algorithms (see Section~\ref{sec:conclusion} for 
further discussion of this point). For much more
information about semidefinite optimization, 
including duality theory, numerical algorithms, and applications,  see for
example~\cite{vandenberghe1996semidefinite}.

\paragraph{Semidefinite reformulations}
Many different optimization problems arising in a variety of contexts,
including some optimization problems for which the natural formulation is not
convex, can be reformulated as semidefinite optimization problems. Given an
optimization problem, by a \emph{semidefinite reformulation} we mean a semidefinite
optimization problem such that
\begin{enumerate}
    \item the optimal value of the semidefinite optimization problem and the original
        optimization problem are the same;
    \item there is an efficient procedure to take an optimal solution to 
        the semidefinite optimization problem and produce an optimal solution to 
        the original optimization problem.
\end{enumerate}

\subsection{Wahba's problem}
\label{sec:wahba-sdp}

We illustrate the basic idea of semidefinite reformulations with the
example of solving Wahba's problem. We note that there are much better 
ways to solve Wahba's problem. The advantage of the semidefinite reformulation is 
that it can be extended to more complicated situations, such as Psiaki's first problem.
The reformulation presented in this section appears (in a more
general context) in~\cite{sanyal2011orbitopes} and is generalized to the analogous
problem where $SO(3)$ is replaced with $SO(n)$ for any $n\geq 2$
in~\cite{saunderson2014semidefinite}. (See also~\cite{forbes2014linear}
where a semidefinite \emph{relaxation} of Wahba's problem is described, as
well as conditions under which it is exact.)

Wahba's problem fits into the general form~\eqref{eq:genform1}
where $A_0 = \sum_{n=0}^{N}\kappa_n y(n\tau)x(n\tau)^T$ and
all the other terms vanish. Using the quaternion parameterization of $SO(3)$, Wahba's problem 
can be expressed as 
\begin{equation}
\label{eq:wahba-original}
\max_{Q\in SO(3)}\langle A_0,Q\rangle = \max_{q\in \H} \langle A_0,\Amap(qq^T)\rangle = \max_{q\in \H} \langle \Amap^*(A_0),qq^T\rangle.
\end{equation}
We now explain how to reformulate~\eqref{eq:wahba-original} 
as a semidefinite optimization problem following a general 
pattern that we use again in Section~\ref{sec:psiaki-sdp}.
\begin{enumerate}
    \item Rewrite the problem as the optimization of a \emph{linear} functional 
        over some set. In this case
        \[ \max_{Z} \langle \Amap^*(A_0),Z\rangle \quad\text{s.t.}\quad 
    Z\in \{qq^T: q\in \H\}.\]
    \item Replace the constraint set with the \emph{convex hull} of the constraint set. In this case
    \[ \max_{Z} \langle \Amap^*(A_0),Z\rangle \quad\text{s.t.}\quad Z \in \conv\{qq^T: q\in \H\}.\]
        This optimization problem has the same optimal value as the original
        non-convex problem because 
        the cost function is linear (see Section~\ref{sec:notation}).
    \item Describe the convex hull of the constraint set as the feasible region of a 
        semidefinite optimization problem (if possible). In this case such a description is well known~(see, e.g.,~\cite[Theorem 3]{overton1992sum}) 
        and given by
    \[ \conv\{qq^T: q\in \H\} = \{Z\in \Sym^n: Z \psd 0,\; \tr(Z) = 1\}.\]
        (This holds because if $Z\psd 0$ and $\tr(Z) =1$ then any eigendecomposition $Z = \sum_{i=1}^{n}\lambda_i q_iq_i^T$
        expresses $Z$ as a convex combination of matrices of the form $qq^T$ with $\|q\|=1$.)
\end{enumerate}
The resulting semidefinite reformulation of Wahba's problem is
\begin{equation}
    \label{eq:wahbasdp}
    \max_{Z} \;\langle \Amap^*(A_0),Z\rangle \quad\text{s.t.}\quad\tr(Z)=1,\;Z\psd 0.
\end{equation}

\paragraph{Extracting an optimal point}
Let $Q$ be an optimal solution of Wahba's problem~\eqref{eq:wahba-original}, and suppose $q$ is a 
corresponding unit quaternion, so that $Q = \Amap(qq^T)$. Then the positive semidefinite matrix 
$Z = qq^T$ is an optimum for the semidefinite reformulation of Wahba's problem~\eqref{eq:wahbasdp}.
All the optima of the semidefinite reformulation of Wahba's problem are convex combinations of 
points of the form $qq^T$ where $\Amap(qq^T)$ is optimal for the original formulation of Wahba's problem.
Under mild assumptions (such as having access to  at least two generic vector measurements) 
Wahba's problem has a unique solution $Q^\star = \Amap(qq^T)$. Whenever Wahba's problem has a unique solution
it follows that the semidefinite reformulation also has a unique solution $Z^\star = qq^T$ and we can recover 
the solution to Wahba's problem from the solution of the semidefinite relaxation by taking $\Amap(Z^\star)$.

\paragraph{Relationship with the $q$-method}
The value of the semidefinite optimization problem~\eqref{eq:wahbasdp} is the largest
eigenvalue of the Davenport matrix $\Amap^*(A_0)$. This can already be seen from~\eqref{eq:wahba-original}
and the fact that $\max_{q\in \H}\langle \Amap^*(A_0),qq^T\rangle = \max_{q\in \H}q^T\Amap^*(A_0)q = \lambda_{\textup{max}}(\Amap^*(A_0))$.
If $q$ is an eigenvector corresponding to the largest eigenvalue of $\Amap^*(A_0)$ then $Z = qq^T$ 
is an optimal solution of the semidefinite reformulation~\eqref{eq:wahbasdp}. As such, 
our reformulation is closely related to the $q$-method for solving Wahba's problem 
problem~\cite{keat1977analysis}.

\paragraph{Discussion}
Note that the transformations in the first and second steps above are merely formal and 
can be applied to essentially any optimization problem. The third step is 
non-trivial. In general it is not well understood which sets $S$ have the property
that $\conv(S)$ can be described as the feasible region of a semidefinite optimization 
problem---this is an area of active research (see, for example,~\cite{blekherman2013semidefinite}). 
One view of this paper is that it shows how to express the convex hulls of the
non-convex constraint sets appearing in certain 
joint spin-rate and attitude estimation problems as the feasible regions of 
semidefinite optimization problems. 

\subsection{Trigonometric Wahba problems}
\label{sec:psiaki-sdp}
We now show how to give semidefinite reformulations of \probclass{} (defined in~\eqref{eq:genform1}). By  
specializing to the case where $(A_n)_{n=0}^{N}$ and $(B_n)_{n=1}^{N}$ are 
given by~\eqref{eq:A0first} and~\eqref{eq:AnBn}, we obtain a 
semidefinite reformulation of Psiaki's first problem.

As in the case of Wahba's problem we use the parameterization of $SO(3)$ 
in terms of unit quaternions to rewrite \probclass{} as
\begin{equation}
    \label{eq:genquat}
        \max_{\substack{q\in \H\\\omega\in [-\pi,\pi)}} \langle \Amap^*(A_0),qq^T\rangle + \sum_{n=1}^{N} \left[\langle \Amap^*(A_n),\cos(n\omega)qq^T\rangle +
                \langle \Amap^*(B_n),\sin(n\omega)qq^T\rangle\right].
     \end{equation}
We can view this problem as the maximization of a linear functional over the set
\begin{equation}
    \label{eq:moment}
    \moment_N:=\{(qq^T,qq^T\cos(\omega),qq^T\sin(\omega),\ldots,qq^T\cos(N\omega),qq^T\sin(N\omega))\in (\Sym^4)^{2N+1}: q\in \H, \omega\in [-\pi,\pi)\}.
    \end{equation}
    As such the convexified version of~\eqref{eq:genquat} is the following optimization problem where the decision variables are 
    the $2N+1$ symmetric matrices $X_0,X_1,Y_1,\ldots,X_N,Y_N$:
\begin{multline}
    \label{eq:genconv}
    \max_{(X_n)_{n=0}^{N},(Y_n)_{n=1}^{N}}\; \langle \Amap^*(A_0),X_0\rangle + \sum_{n=1}^{N}\left[\langle \Amap^*(A_n),X_n\rangle + \langle \Amap^*(B_n),Y_n\rangle\right]\\
    \textup{subject to} \quad (X_0,X_1,Y_1,\ldots,X_N,Y_N)\in \conv(\moment_N).
\end{multline}
This problem is certainly convex, and has the same optimal value as~\eqref{eq:genform1} and~\eqref{eq:genquat}. It may not be 
immediately clear that the constraint set $\conv(\moment_N)$ has a succinct representation in terms of the feasible region 
of a semidefinite optimization problem. In fact $\conv(\moment_N)$ does have such a representation, and we now turn our attention 
to describing it. 
\paragraph{A linear matrix inequality description of $\conv(\moment_N)$}
We now describe $\conv(\moment_N)$ in terms of a linear matrix inequality, 
making use of the block matrix notation defined in Section~\ref{sec:notation}.
We establish the correctness of this description in the Appendix,
by combining standard results with a novel symmetry reduction argument.
\begin{proposition}
\label{prop:symm}
    \begin{multline} 
            \label{eq:symm}
            \conv(\moment_N) = \{(X_0,X_1,Y_1,\ldots,X_N,Y_N)\in (\Sym^4)^{2N+1}: \;\tr(X_0) = 1, \\
                \BlkToep(X_0,X_1,\ldots,X_N) + \BlkHank(Y_N,Y_{N-1},\ldots,Y_1,0,-Y_1,\ldots,-Y_{N-1},-Y_N) \psd 0\}.
    \end{multline}
\end{proposition}
\begin{proof}
    We provide a proof in the Appendix.
\end{proof}

\paragraph{Semidefinite reformulation in the general case}
Now that we have a semidefinite description of $\conv(\moment_N)$, we can give a semidefinite reformulation for all \probclass{}.
The following theorem explicitly describes this reformulation, which is obtained by 
replacing $\conv(\moment_N)$ in~\eqref{eq:genconv} with its semidefinite description from Proposition~\ref{prop:symm}.
\begin{thm}
    \label{thm:mainsdp}
    Let $A_0,A_1,\ldots,A_N,B_1,\ldots,B_N\in \R^{3\times 3}$. Then the \probinstance{}
\begin{equation}
    \label{eq:twthm}\max_{\substack{Q\in SO(3)\\\omega\in [-\pi,\pi)}}
    \langle A_0,Q\rangle +
    \sum_{n=1}^{N}\left[ \langle A_n,\cos(\omega n)Q\rangle + 
        \langle B_n,\sin(\omega n)Q\rangle\right] 
\end{equation}
    and the semidefinite optimization problem
    \begin{align}
    \max_{(X_n)_{n=0}^{N},(Y_n)_{n=1}^{N}}\;&\langle \Amap^*(A_0),X_0\rangle + \sum_{n=1}^{N}[\langle \Amap^*(A_n),X_n\rangle +
    \langle \Amap^*(B_n),Y_n\rangle]\label{eq:mainsdp}\\
    \textup{s.t.} &\;\; 
        \BlkToep(X_0,X_1,\ldots,X_N) + \BlkHank(Y_N,Y_{N-1},\ldots,Y_1,0,-Y_1,\ldots,-Y_{N-1},-Y_N) \psd 0\nonumber\\
        & \;\; \tr(X_0) = 1\nonumber
    \end{align}
    have the same optimal value. The set of optimal points of the semidefinite reformulation is
\begin{align}
\conv&\,\{(qq^T,qq^T\cos(\omega),qq^T\sin(\omega),\ldots,qq^T\cos(N\omega),qq^T\sin(N\omega)):\label{eq:sols}\\
     &\qquad\qquad\qquad\qquad\qquad\qquad\qquad \text{$(\omega,\Amap(qq^T))$ is an optimal point for~\eqref{eq:twthm}}\}\nonumber.
\end{align}
\end{thm}
\paragraph{Extracting an optimal solution}
If $N\geq 2$ we expect a generic \probinstance{} to have a unique optimal 
point~$(\omega^\star,Q^\star)$~\cite{psiaki2009generalized}. In that case the 
    semidefinite reformulation~\eqref{eq:mainsdp} has a unique optimal point denoted
    $(X_0^\star,X_1^\star,Y_1^\star,\ldots,X_N^\star,Y_N^\star)$ from which we can recover $(\omega^\star,Q^\star)$ via
    \begin{equation}
        \label{eq:relation}
        Q^\star = \Amap(X_0^\star),\quad\cos(\omega^\star) = \tr(X_1^\star)\quad\text{and}\quad
            \sin(\omega^\star) = \tr(Y_1^\star).
        \end{equation}

    

\subsection{Pseudocode}
\label{sec:code}

In this section we describe code to implement our semidefinite
optimization-based formulations~\eqref{eq:mainsdp}
of \probclass{}. Our motivation for doing this is to show that it
is quite straightforward to use standard 
numerical routines to solve the semidefinite optimization problems
that appears in this paper. 

The code is expressed in a parsing language called
YALMIP~\cite{yalmip} that runs under MATLAB. Internally, 
YALMIP reformulates the human-readable description of the 
optimization problem we specify into a standard format, then calls 
a numerical solver for semidefinite optimization problems 
(we used MOSEK~\cite{andersen2000mosek} version 7 for these 
experiments) to solve the optimization problem. 

In what follows, we assume we have functions
\begin{description}
    \item[\underline{\small\texttt{A\_map}}] implementing the linear 
        map $\Amap$  taking a $4\times 4$ symmetric matrix and 
        returning a $3\times 3$ matrix according to~\eqref{eq:amap};
    \item[\underline{\small{\texttt{block\_toeplitz}}}] implementing the linear map 
        $(X_0,X_1,\ldots,X_N)\mapsto \BlkToep(X_0,X_1,\ldots,X_N)$ taking a 
        $4\times 4 \times (N+1)$ array and returning a $4(N+1)\times 4(N+1)$ 
        matrix according to~\eqref{eq:blktoep};
    \item[\underline{\small{\texttt{block\_hankel}}}] implementing the linear map 
        $(Y_1,Y_2,\ldots,Y_N) \mapsto \BlkHank(-Y_N,\ldots,-Y_1,0,Y_1,\ldots,Y_N)$
        taking a $4\times 4 \times N$ array and returning a $4(N+1)\times 4(N+1)$
        matrix according to~\eqref{eq:blkhank}.
\end{description}

We declare variables in YALMIP using the {\small\texttt{sdpvar}} command. 
\begin{align*}
    \mbox{\small\texttt{1:}}\qquad& \mbox{\small\texttt{X = sdpvar(4,4,N+1,'symmetric');}}\\
    \mbox{\small\texttt{2:}}\qquad& \mbox{\small\texttt{Y = sdpvar(4,4,N,'symmetric');}}\\
    \intertext{For example {\small\texttt{Y}} is a $4\times 4 \times N$ array of variables
        with each slice {\small\texttt{Y(:,:,n)}} being a symmetric matrix. 
        We specify constraints by constructing an array of constraints expressed 
        in a very natural way. We express the two constraints in~\eqref{eq:mainsdp} by}
        \mbox{\small\texttt{3:}}\qquad& \mbox{\small\texttt{K = [trace(X(:,:,1))==1, \underline{block\_toeplitz}(X) + \underline{block\_hankel}(Y) >= 0];}}\\
    \intertext{where we have indexed from 1 following MATLAB's conventions. Note that in YALMIP this latter inequality is automatically 
interpreted in the positive semidefinite sense since the matrix on the left hand side is structurally symmetric.}
\intertext{Suppose the variables {\small\texttt{A}} and {\small\texttt{B}} are respectively $4\times 4 \times (N+1)$
    and $4\times 4 \times N$ arrays with {\small\texttt{A(:,:,n+1)}} being $\Amap^*(A_{n})$ and
    {\small\texttt{B(:,:,n)}} being $\Amap^*(B_n)$. 
        Then we can solve the semidefinite optimization problem~\eqref{eq:mainsdp} with the single line}
        \mbox{\small\texttt{4:}}\qquad&\mbox{\small\texttt{solvesdp(K,-(A(:)'*X(:)\;+\;B(:)'*Y(:)));}}\\
        \intertext{which calls a numerical solver with the constraint set {\small\texttt{K}} and  the cost function
            {\small\texttt{-(A(:)'*X(:) + B(:)'*Y(:))}} (with the minus sign because minimization is the default). Assuming that there is a unique solution to the 
     non-convex problem we can extract the optimal rotation matrix $Q$ and optimal $\omega$ with} 
     \mbox{\small\texttt{5:}}\qquad& \mbox{\small\texttt{Q\_opt = \underline{A\_map}(double(X(:,:,1)));}}\\
     \mbox{\small\texttt{6:}}\qquad&\mbox{\small\texttt{omega\_opt = atan2(trace(double(Y(:,:,1))),trace(double(X(:,:,2))));}}
 \end{align*}

\section{Variations}
\label{sec:variations}
In Section~\ref{sec:semidefinite} we formulated Psiaki's first problem 
as a semidefinite optimization problem by showing how to 
express the convex hull of $\moment_{N}$ in terms of 
linear matrix inequalities. This description of $\moment_{N}$
also allows us to take a semidefinite optimization-based 
approach to many variations on Psiaki's first problem. 
In this section we illustrate the possibilities in this direction
with one simple example---
a variant on Psiaki's problem where we assume the measurement
errors are bounded, and incorporate this additional information
into the formulation.

\subsection{Psiaki's first problem with bounded measurement errors}
\label{sec:psiakivar}
Using the notation from Section~\ref{sec:jas}, 
suppose we know that the error between the measured direction $y_n$ and 
the true direction $Q(n \tau)x_n$ is bounded in each coordinate, satisfying
\begin{equation}
    \label{eq:box}
    -\epsilon \leq y_n - Q(n\tau)x_n \leq \epsilon.
\end{equation}
Here $\epsilon = \begin{bmatrix} \epsilon_1 & \epsilon_2 &\epsilon_3\end{bmatrix}^T$ 
is a vector of positive constants that are not necessarily equal, and the inequalities
in~\eqref{eq:box} are to be interpreted element-wise.
Adding these constraints to the formulation of Psiaki's first problem 
we obtain the following variant:
\begin{align}
    \min_{\substack{Q(0)\in SO(3)\\\omega\in [-\pi/\tau,\pi/\tau)}} &
        \sum_{n=0}^{N}\frac{\kappa_n}{2}\|y_n -
    Q(n\tau)x_n\|_2^2\label{eq:psiakivar-lsq}\\
    \text{s.t.} & \quad -\epsilon \leq y_n - Q(n\tau)x_n \leq \epsilon
    \quad\text{for $n=0,1,\ldots,N$}.\nonumber
\end{align}
Here, as in Section~\ref{sec:jas}, $Q(n\tau)$ is related 
to $Q(0)$ via~\eqref{eq:spin2} and so putting $\omega' = \omega \tau$, 
\begin{align*}
Q(n\tau) & = \begin{bmatrix} 1 & 0 & 0\\0 & \cos(n\omega') & -\sin(n\omega')\\0 & \sin(n \omega') & \cos(n \omega')\end{bmatrix}Q\\
         & = 
\left[\begin{smallmatrix} 1 & 0 & 0\\0 & 0 & 0\\0 & 0 & 0\end{smallmatrix}\right]\Amap(qq^T)
+ \left[\begin{smallmatrix} 0 & 0 & 0\\0 & 1 & 0\\0 & 0 & 1\end{smallmatrix}\right]\Amap(qq^T\cos(n\omega')) + 
\left[\begin{smallmatrix} 0 & 0 & 0\\0 & 0 & -1\\0 & 1 & 0\end{smallmatrix}\right]\Amap(qq^T\sin(n\omega')).
\end{align*}
Since the objective function of~\eqref{eq:psiakivar-lsq} is identical to the objective function of Psiaki's first problem~\eqref{eq:psiaki-lsq}, using the notation 
and manipulations of Sections~\ref{sec:jas} and~\ref{sec:semidefinite}
we can rewrite the variant of Psiaki's first problem as
\begin{align}
    \max_{\substack{q\in \H\\\omega' \in [-\pi,\pi)}} &
    \langle \Amap^*(A_0),qq^T\rangle + \sum_{n=1}^{N}[\langle \Amap^*(A_n),qq^T\cos(n\omega')\rangle + \langle \Amap^*(B_n),qq^T\sin(n\omega')\rangle]\\ 
\label{eq:psiakivar-linear}\\
\textup{s.t.}-\epsilon \leq& y_n - \left(\left[\begin{smallmatrix} 1 & 0 & 0\\0 & 0 & 0\\0 & 0 & 0\end{smallmatrix}\right]\!\Amap(qq^T)
+ \left[\begin{smallmatrix} 0 & 0 & 0\\0 & 1 & 0\\0 & 0 & 1\end{smallmatrix}\right]\!\Amap(qq^T\cos(n\omega')) + 
\left[\begin{smallmatrix} 0 & 0 & 0\\0 & 0 & -1\\0 & 1 & 0\end{smallmatrix}\right]\!\Amap(qq^T\sin(n\omega'))\right)x_n\leq \epsilon\label{eq:varconst}\\
                            &\qquad\qquad\qquad\qquad\qquad\qquad\qquad\qquad\qquad\qquad\qquad\qquad\qquad\qquad\text{for $n=0,1,\ldots,N$.}\nonumber\end{align}
 
Observe that we have rewritten the problem as the maximization of a linear functional over the constraint set defined by
\begin{align*}
    S & = \bigg\{(qq^T,qq^T\cos(\omega),qq^T\sin(\omega),\ldots,qq^T\cos(N\omega),qq^T\sin(N\omega)):q\in \H, \omega \in [-\pi,\pi), \\
      &\quad-\epsilon \leq y_n - \left(\left[\begin{smallmatrix} 1 & 0 & 0\\0 & 0 & 0\\0 & 0 & 0\end{smallmatrix}\right]\Amap(qq^T)
+ \left[\begin{smallmatrix} 0 & 0 & 0\\0 & 1 & 0\\0 & 0 & 1\end{smallmatrix}\right]\Amap(qq^T\cos(n\omega)) + 
\left[\begin{smallmatrix} 0 & 0 & 0\\0 & 0 & -1\\0 & 1 & 0\end{smallmatrix}\right]\Amap(qq^T\sin(n\omega))\right)x_n\leq \epsilon\\
                                                 &\qquad\qquad\qquad\qquad\qquad\qquad\qquad\qquad\qquad\qquad\qquad\qquad\qquad\qquad\text{for $n=0,1,\ldots,N$}\bigg\}.
           \end{align*}
           This set $S$ is the intersection of $\moment_{N}$ with the additional constraints~\eqref{eq:varconst} coming 
           from incorporating the knowledge that the measurement errors 
           satisfy the explicit deterministic bounds described
           in~\eqref{eq:box}.
\subsection{A semidefinite relaxation}
\label{sec:varsdp}
Recall from Section~\ref{sec:semidefinite} that if we could exactly describe 
$\conv(S)$ in terms of linear matrix inequalities that are not too large, 
we could obtain a semidefinite reformulation of this problem that can be 
solved efficiently. Unfortunately we do not know of such a concise
description of $\conv(S)$, and conjecture that no such concise
description exists for all choices of the $x_n$ and $y_n$. 
 
Instead, a natural general approach is to construct a \emph{semidefinite relaxation} of $\conv(S)$. By this we mean
a convex set $C$ such that 
\begin{enumerate}
    \item $C \supseteq \conv(S)\supseteq S$ and
    \item $C$ has a simple description in terms of linear matrix inequalities.
\end{enumerate}
One choice would be to take $C$ to be the convex set
\begin{align*} 
    C &= \bigg\{(X_0,X_1,Y_1,\ldots,X_N,Y_N)\in \conv\,\moment_N:\\
      &
-\epsilon \leq y_n - \left(\left[\begin{smallmatrix} 1 & 0 & 0\\0 & 0 & 0\\0 & 0 & 0\end{smallmatrix}\right]\Amap(X_0)
+ \left[\begin{smallmatrix} 0 & 0 & 0\\0 & 1 & 0\\0 & 0 & 1\end{smallmatrix}\right]\Amap(X_n) + 
\left[\begin{smallmatrix} 0 & 0 & 0\\0 & 0 & -1\\0 & 1 & 0\end{smallmatrix}\right]\Amap(Y_n)\right)x_n\leq \epsilon 
\quad\text{for $n=0,1,\ldots,N$}\bigg\}.
\end{align*}
One can check that while $C$ is, in general, strictly larger than $\conv(S)$, it can be expressed using linear matrix inequalities (since $C$ is 
obtained by adding linear inequalities to $\conv(\moment_N)$ which has a linear matrix inequality description from Proposition~\ref{prop:symm}). 

By optimizing over $C$ rather than $S$
we obtain the following \emph{semidefinite relaxation} of the optimization problem
\begin{align}
    \max_{(X_n)_{n=0}^{N},(Y_n)_{n=1}^{N}}&\;\langle \Amap^*(A_0),X_0\rangle  + \sum_{n=0}^{N}[\langle \Amap^*(A_n),X_n\rangle + \langle \Amap^*(B_n),Y_n\rangle]
\label{eq:varsdp}\\
\quad\text{s.t.}&\quad (X_0,X_1,Y_1,\ldots,X_N,Y_N)\in C.\nonumber
\end{align}
When we solve this semidefinite relaxation, if the solution $(X_0^\star,X_1^\star,Y_1^\star,\ldots,X_N^\star,Y_N^\star)\in C$, returned 
by the solver, is actually in $S$, then it is a solution of the original non-convex problem~\eqref{eq:psiakivar-linear}
we are trying to solve. In this case it is typical to say that the semidefinite relaxation 
is \emph{exact} for this instance. 

If $(X_0^\star,X_1^\star,Y_1^\star,\ldots,X_N^\star,Y_N^\star) \notin S$, we have not solved the non-convex problem, 
but can still conclude that 
the value of the objective function at this point is an upper bound on the 
optimal value of the original non-convex maximization problem~\eqref{eq:psiakivar-linear}. 
Such a bound can be used, for example, to assess the quality (in terms of 
the objective function) of any feasible point obtained, for instance,
by a local optimization method.
 
\subsection{Numerical experiments}
\label{sec:expts}
In this section we describe the results of two simple numerical experiments to illustrate solving  
Psiaki's first problem using semidefinite optimization, as well as solving the semidefinite 
relaxation of the variant on Psiaki's problem discussed in Sections~\ref{sec:psiakivar} and~\ref{sec:varsdp}. 
 
For all experiments we use the same parameters as in Psiaki's truth-model simulation 
in~\cite{psiaki2009generalized}---the true spin period is $45.32$ seconds 
(so the true spin-rate is $\omega = 0.1386$ radians per second),
the sampling period is $\tau = 7.7611$ seconds per sample, 
and the initial attitude is $Q(0) = I$. The attitude dynamics are described by~\eqref{eq:spineq}.

\paragraph{Solving Psiaki's first problem}
In the first experiment we solve Psiaki's first problem using the semidefinite 
reformulation~\eqref{eq:mainsdp}. In particular we repeat the 
following experiment $T=1000$ times:
\begin{enumerate}
    \item Sample reference directions $x_0,x_1,\ldots,x_{10}$ uniformly on the sphere.
    \item For $n=0,1,2\ldots,N$, sample measurements $y_n$ 
        uniformly distributed on the intersection of the 
        unit sphere and the region
    \begin{equation}
    \label{eq:boxexpt}
-\epsilon \leq y_n - Q(t_n)x_n \leq \epsilon
\end{equation}
    where $\epsilon = \begin{bmatrix}0.5 & 0.5 & 0.05\end{bmatrix}^T$ (by sampling uniformly on the sphere and 
        rejecting those samples not in the box-shaped region). 
        This corresponds to measurements that
        are very accurate along one axis, but quite inaccurate in other directions. 
    \item For $N=2,3,\ldots,10$, use the reference directions $x_0,x_1,\ldots,x_N$ and measurements
        $y_0,y_1,\ldots,y_N$ and solve
        the semidefinite optimization reformulation of Psiaki's first problem~\eqref{eq:mainsdp}.
\end{enumerate}
We note that although the data are generated from a model where the measurement errors satisfy
the explicit bounds~\eqref{eq:boxexpt}, we do not exploit this in our solution method. Also, to get a sense of the measurement 
errors introduced, we note that under the noise model we have adopted, the angle 
between $y_n$ and $Q(t_n)x_n$ over all samples was at most $41.1$ degrees, and on average $16.8$ degrees.

The average angular error (in degrees) between the estimate of the initial attitude and the true initial attitude 
is indicated by cross-shaped markers in Figure~\ref{fig:Qopterr}. Given an estimate $\hat{Q}$ of the true initial attitude
$Q(0) = I$, the angular error $\theta$ satisfies $\textup{tr}(\hat{Q}^TQ(0)) = 2\cos(\theta)+1$. Hence
we compute the angular error via $|\cos^{-1}[\textup{tr}(\hat{Q}^TQ(0))-1)/2]|$. The corresponding average error in 
the spin-rate estimate $\hat{\omega}$ is computed by taking the mean of $|\hat{\omega} - \omega|$ over all trials
and is indicated by cross-shaped markers in Figure~\ref{fig:omegaerr}.
It is clear that as more vector measurements are used (i.e.\ as $N$ increases) 
the estimates improve, justifying using more than the minimum number of measurements required for the 
optimization problem to have a unique optimum. 

\paragraph{Solving the variant with bounded measurement errors}
In the second experiment we use exactly the same data as for the first experiment. This time, 
instead of solving the semidefinite reformulation of Psiaki's first problem, we solve the 
semidefinite relaxation~\eqref{eq:varsdp} of the variant of Psiaki's first problem with bounded
measurement errors. This estimation method explicitly makes use of the fact that the measurement
errors satisfy~\eqref{eq:boxexpt}. 

\begin{figure}[h]
    \begin{center}
        \subfigure[Angular errors in the initial attitude]{\label{fig:Qopterr}
        \includegraphics[trim = 52mm 172mm 90mm 40mm, clip]{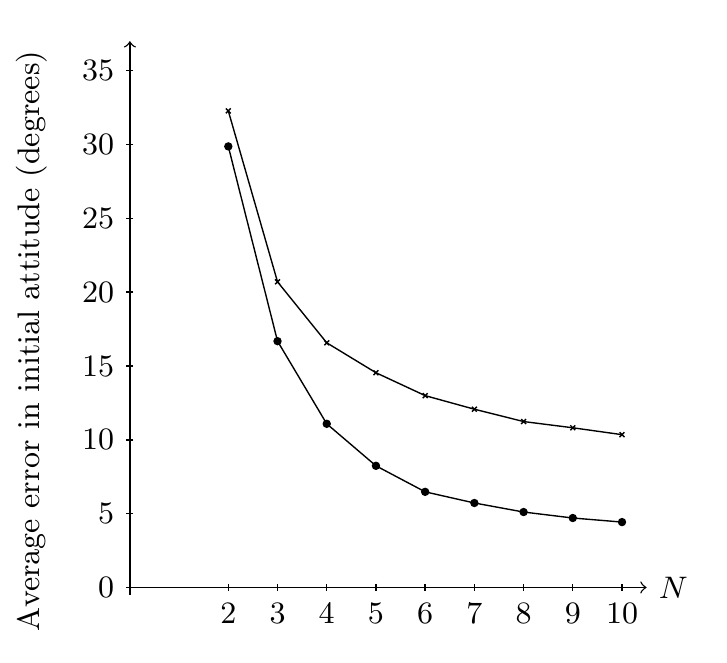}
        }
        \subfigure[Error in the spin-rate estimate]{\label{fig:omegaerr}
            \includegraphics[trim = 52mm 172mm 90mm 40mm, clip]{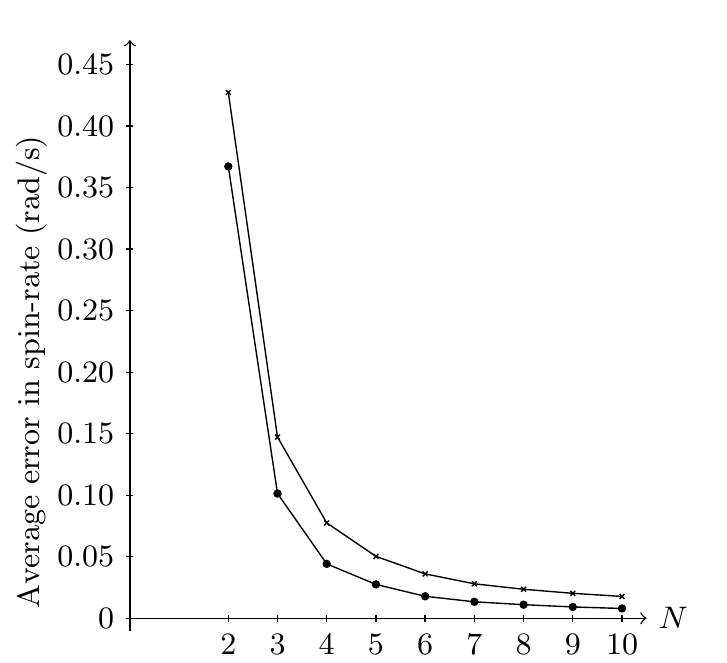}
        }
    \end{center}
    \caption{\label{fig:exptresults}Results from the experiment described in Section~\ref{sec:expts}. Figure~\ref{fig:Qopterr}
    shows the average (over $1000$ random trials) error in estimating the initial attitude using $N+1$ measurements
    by solving Psiaki's first problem (cross-shaped markers) and the semidefinite relaxation of the bounded error variant described in Section~\ref{sec:varsdp} 
    (dot-shaped markers).
Similarly Figure~\ref{fig:omegaerr} shows the average error in estimating the spin-rate for the same experiment.}
\end{figure}

As discussed in Section~\ref{sec:varsdp}, and unlike the case for the semidefinite 
reformulation of Psiaki's first problem itself,  when solving the semidefinite relaxation of the bounded error 
variant of Psiaki's first problem there is no guarantee that the relaxation will be exact. In other words, 
we do not know, in advance, whether the solution to the semidefinite optimization problem corresponds to 
a solution of the original non-convex problem~\eqref{eq:psiakivar-linear}. On the other hand, after solving 
the semidefinite optimization problem we can determine whether or not this is the case by checking whether
the solution is feasible for the original problem. The 
number of trials for which the semidefinite relaxation was exact is 
listed in Table~\ref{tab:exactness} for each $N=2,3,\ldots,10$. 

For this experiment, the average angular error in the initial attitude estimates and the average error in the spin-rate 
estimates are indicated by dot-shaped markers in Figures~\ref{fig:Qopterr} and~\ref{fig:omegaerr}.
If the semidefinite relaxation was exact we compute the errors in these quantities in the same 
way as for the previous experiment. If the semidefinite relaxation was not exact, to be as conservative
as possible we take the error to be the maximum possible value: $180$ degrees for the initial attitude error and 
$\pi$ radians/second for the spin-rate error.

The results in Figures~\ref{fig:Qopterr} and~\ref{fig:omegaerr} show that by incorporating explicit  
bounds on the measurement errors (if known) into the semidefinite optimization framework, significantly 
better estimates of the initial attitude can be obtained. Furthermore Table~\ref{tab:exactness}
indicates that the semidefinite relaxation was indeed exact on many of our random trials, suggesting 
that the semidefinite relaxation approach may be well suited to tackling
at least this variant on Psiaki's first problem, and perhaps others.

\begin{table}
    \begin{center}
    \begin{tabular}{c|ccccccccc}
        $N$ & 2 & 3 & 4 & 5 & 6 & 7 & 8 & 9 & 10\\\hline
        $T_\textup{exact}$ (/1000) & 842 & 816 & 867 & 918 & 948 & 958 & 965 & 969& 973
    \end{tabular}
\end{center}
\caption{\label{tab:exactness}$T_\textup{exact}$ is the number of random experiments (as described in Section~\ref{sec:expts})
    for which the semidefinite relaxation for the variant on Psiaki's first problem~\eqref{eq:varsdp}
    was exact. Here $N+1$ is the number of vector measurements used.}
\end{table}

\section{Future directions}
\label{sec:directions}

We briefly comment on possible future research directions based on the work in the present paper.

\subsection{Numerical algorithms}

The semidefinite reformulation of Psiaki's first problem~\eqref{eq:mainsdp} has
a very specific structure. This structure could be exploited to develop
numerical algorithms for its solution (as well as the solution for variants on
the problem) that are much faster than the generic interior-point algorithms we
used for our experiments.  Indeed semidefinite optimization problems with a
similar structure arise in problems related to the Kalman-Yakubovich-Popov (KYP) 
lemma in robust control and in that context numerous specialized algorithms have been developed for
their solution (see, e.g.,~\cite{liu2007low,genin2003optimization,kao2007new}).
Furthermore, great gains can be made by producing optimized low-level code for
a particular family of convex optimization problems.  An excellent example of
this is the code-generation software CVXGEN which focuses on linear and convex
quadratic programs~\cite{mattingley2012cvxgen}. 

\subsection{Further variants}
\label{sec:fvariations}
A semidefinite reformulation of a problem is particularly useful because it can
be combined in many ways with other semidefinite optimization primitives to
yield more problems that can be solved in the semidefinite optimization
framework. In Section~\ref{sec:variations} we discussed a variation on Psiaki's
first problem that had bounds on the angular noise. In this case it was
straightforward to formulate a semidefinite relaxation using our semidefinite
reformulation of Psiaki's first problem. 

Another natural variation that could be approached this way would be to obtain
semidefinite relaxations of Psiaki's first problem that are robust to
uncertainty in certain model parameters. A similar idea has been carried out in
detail for Wahba's problem by Ahmed et al.~\cite{ahmed2012semidefinite}.  They
extended the semidefinite formulation of Wahba's
problem~\cite{sanyal2011orbitopes} to a variant that is robust to uncertainty
in certain parameters, such as the reference directions.  As suggested by Ahmed
at al.~this could be useful when using magnetometer measurements together with
a low-order magnetic field model.

\subsection{Psiaki's second problem}

It would be interesting to try to take a similar approach to the one taken in
the present paper to related problems, such as Psiaki's second problem. To do
so we would need to give a semidefinite description (or perhaps a relaxation)
of
\begin{equation}
    \label{eq:convP2} \conv_{Q(t_0)\in SO(3),\Omega(t_0)}\,
    \{(Q(t_0),Q(t_1),\ldots,Q(t_N)): Q(t) =
        \Phi(t-t_0;\Omega(t_0))Q(t_0)\quad\text{for all $t\geq t_0$}\}.
\end{equation}
Given this, we note that the objective function~\eqref{eq:psiaki-lsq} can be
rewritten as the maximization of a linear functional over the convex hull
described in~\eqref{eq:convP2}.

A more modest goal along similar lines might be to discretize the differential
equation~\eqref{eq:ode} and try to compute the convex hull (over all initial
conditions $Q(t_0),\Omega(t_0)$) of an appropriately subsampled trajectory of
the associated difference equation for the attitude variables. 
This approach of convexifying a problem based on discretized dynamics
would, in a sense, be a convex analogue of the methods proposed 
for Psiaki's second problem in~\cite{psiaki2012numerical}.

\section{Conclusion}
\label{sec:conclusion}
We have shown how Psiaki's generalization of Wahba's problem to the case of a
spacecraft spinning around a fixed axis at an unknown rate can be exactly
reformulated as a semidefinite optimization problem.  Such convex optimization
problems can be solved globally using standard methods for semidefinite
optimization. As suggested by Psiaki when formulating his generalizations of
Wahba's problem~\cite{psiaki2009generalized}, our solutions to these
generalizations of Wahba's problem could be used to initialize standard
extended Kalman filter-based methods for attitude estimation.

Furthermore, we have illustrated how to use our reformulation of
Psiaki's first problem to construct semidefinite relaxations of a more
complicated variant on the problem. Our numerical experiments with a variant
that includes explicit bounds on the measurement errors suggest that
incorporating additional information into the formulation can improve the
estimation errors. Our results also suggest the semidefinite relaxation
approach we propose, although not exact in general, often computes solutions to
the original non-convex variations of Psiaki's problem that we ultimately are
aiming to solve.
 
\appendix

\section{Proofs}
\label{app:pfs}

In this appendix we prove Proposition~\ref{prop:symm}. We split the proof into
two parts, given by Lemmas~\ref{lem:herm} and~\ref{lem:realify} below.
Together these clearly imply Proposition~\ref{prop:symm}.
In what follows we extend the notation $\BlkToep(T_0,T_1,\ldots,T_N)$ defined
in~\eqref{eq:blktoep} to include the case where $T_1,\ldots,T_N$ are $d\times
d$ complex matrices, $T_0$ is a $d\times d$ Hermitian matrix, and all
transposes of real matrices are replaced with conjugate transposes, denoted $A\mapsto A^*$,
of complex matrices.

Lemma~\ref{lem:herm}, to follow,
is a slight modification of the fact that any Hermitian positive semidefinite
block-Toeplitz matrix admits a decomposition as a sum of rank one positive
semidefinite block-Toeplitz matrices. This fact may be more familiar in its dual form as
the matrix spectral factorization (or Fej\'{e}r-Riesz) theorem~(see, e.g.,~\cite{rosenblatt1958multi}). 
This classical result says that any Hermitian matrix-valued function $S(e^{i\omega}) = \sum_{n=-N}^{N}S_ne^{in\omega}$
that is positive semidefinite for all $\omega$ has a factorization as 
$S(e^{i\omega}) = W(e^{i\omega})^*W(e^{i\omega})$ where $W(e^{i\omega})$ has the form
$W(e^{i\omega}) = \sum_{n=0}^{N}W_n e^{in\omega}$. This result can also be interpreted 
as saying that non-negative functions of the form $(z,\omega)\mapsto z^*S(e^{i\omega})z$, with $S(e^{i\omega})$ as before, 
are sums-of-squares~\cite{aylward2007explicit}.
\begin{thm}[Tismenetsky~\cite{tismenetsky1993matrix}]
    \label{thm:extrays}
If $\BlkToep(T_0,T_1,\ldots,T_N) \pd 0$ then there are $u_k\in \C^4$,
$\omega_{k}\in [-\pi,\pi)$ and $\lambda_k > 0$ for $k=1,2,\ldots,4(N+1)$ such
    that 
\[ \BlkToep(T_0,T_1,\ldots,T_N) = \sum_{k=1}^{4(N+1)}\lambda_k
    \BlkToep(u_k u_k^*,u_ku_k^*e^{i\omega_k},\ldots,u_ku_k^*e^{iN\omega_k}).\]
Consequently $T_j = \sum_{k=1}^{4(N+1)}\lambda_k u_ku_k^*e^{ij\omega_k}$ for
$j=0,1,\ldots,N$. 
\end{thm}
The following lemma is a slight modification of Theorem~\ref{thm:extrays}.
\begin{lemma}
    \label{lem:herm}
    Let $\moment_N$ be defined as in~\eqref{eq:moment}. Then 
    \begin{align}
        \conv(\moment_N) = \{(X_0,X_1,Y_1,\ldots,X_N,Y_N)\in 
            & (\Sym^4)^{2N+1}:\;\tr(X_0)=1\nonumber\\
            & \BlkToep(X_0,X_1+iY_1,\ldots,X_N+iY_N) \psd 0\}.\label{eq:herm}
    \end{align}
\end{lemma}
\begin{proof}
    Let
    $(qq^T,qq^T\cos(\omega),qq^T\sin(\omega),\ldots,qq^T\cos(N\omega),qq^T\sin(N\omega))\in
    \moment_N$. Then $\tr(qq^T)=\|q\|^2 = 1$ and it is straightforward to check
    that 
\[ \BlkToep(qq^T,qq^T\cos(\omega)+iqq^T\sin(\omega),\ldots,qq^T\cos(N\omega)+iqq^T\sin(N\omega)) = 
\begin{bmatrix}q\\qe^{-i\omega}\\\vdots\\ qe^{-i\omega N}\end{bmatrix}
\begin{bmatrix} q \\ qe^{-i\omega} \\ \vdots \\ qe^{-iN\omega}\end{bmatrix}^* \psd 0.\]
    Hence $\moment_N$ is a subset of the right hand side of~\eqref{eq:herm}.
    Since the right hand side of~\eqref{eq:herm} is convex, it follows that
    $\conv(\moment_N)$ is also a subset of the right-hand side
    of~\eqref{eq:herm}.

    Now suppose $(X_0,X_1,Y_1,\ldots,X_N,Y_N)\in (\Sym^4)^{2N+1}$ satisfies
    \[ \tr(X_0)=1\quad\text{and}\quad\BlkToep(X_0,X_1+iY_1,\ldots,X_N+iY_N) \pd 0.\] 
Then by Theorem~\ref{thm:extrays} there are $u_k\in \C^4$, $\omega_k\in
[-\pi,\pi)$ and $\lambda_k > 0$ for $k=1,2,\ldots,4(N+1)$ such that $1 =
    \tr(X_0) = \sum_{k=1}^{4(N+1)}\lambda_k\|u_k\|^2$ and  $X_j+iY_j =
    \sum_{k=1}^{4(N+1)}\lambda_k u_ku_k^*e^{ij\omega_k}$ for $j=0,1,\ldots,N$.
    Since $X_j^T = X_j$ and $Y_j^T=Y_j$ for all $j$, it follows by a
    straightforward calculation that there are $v_k\in \R^4$ and $\lambda_k'>0$
    for $k=1,2,\ldots,8(N+1)$, such that 
\[    X_j+iY_j = \frac{1}{2}\left((X_j+iY_j)+(X_j+iY_j)^T\right)
    = \sum_{k=1}^{4(N+1)}\frac{\lambda_k}{2} \left((u_ku_k^*)+(u_ku_k^*)^T\right)e^{ij\omega_k}
    = \sum_{k=1}^{8(N+1)}\lambda_{k}'v_kv_k^Te^{ij\omega_k}.\]
Defining $\mu_k = \lambda_k'\|v_k\|^2>0$ and $q_k = v_k/\|v_k\|\in \H$ for
$k=1,2,\ldots,8(N+1)$ we have that $\sum_{k=1}^{8(N+1)}\mu_k = 1$ and 
\[ X_j = \sum_{k=1}^{8(N+1)}\mu_k q_kq_k^T\cos(j\omega_k)\quad\text{for $j=0,1,\ldots,N$ and}\quad
    Y_j = \sum_{k=1}^{8(N+1)}\mu_k q_kq_k^T\sin(j\omega_k)\]
for $j=1,2,\ldots,N$.  This shows that $(X_0,X_1,Y_1,\ldots,X_N,Y_N)\in
\conv(\moment_N)$.  Hence the relative interior of the right-hand side
of~\eqref{eq:herm} is a subset of $\conv(\moment_N)$. Since $\conv(\moment_N)$
is closed the right-hand side of~\eqref{eq:herm} is also a subset of
$\conv(\moment_N)$, establishing the result.
\end{proof}

\begin{lemma}
    \label{lem:realify}
    If $X_0,X_1,Y_1,\ldots,X_N,Y_N\in \Sym^d$ then
    \[
    \BlkToep(X_0,X_1+iY_1,\ldots,X_N+iY_N) \psd 0\]
    if and only if 
    \[ 
        \BlkToep(X_0,X_1,\ldots,X_N) + \BlkHank(-Y_N,-Y_{N-1},
        \ldots,-Y_1,0,Y_1,\ldots,Y_{N-1},Y_N) \psd 0.\]
\end{lemma}
\begin{proof}
    First observe that $Z = \BlkToep(X_0,X_1+iY_1,\ldots,X_N+iY_N)\psd 0$
    if and only if the real $2d(N+1)\times 2d(N+1)$ symmetric matrix 
    \[ Z_{\R} = \begin{bmatrix} \Re Z & \Im Z\\-\Im Z & \Re Z\end{bmatrix}\]
        is positive semidefinite~\cite{goemans2001approximation}. Here $\Re Z$
        and $\Im Z$ are the real and imaginary parts of $Z$, respectively.
        Indeed
        \[\Re Z = \begin{bmatrix} X_0 & X_1 & X_2 & \cdots & X_N\\
        X_1 & X_0 & X_1 & \ddots & \vdots\\
        X_2 & X_1 & \ddots & \ddots & \vdots\\
        \vdots & \ddots & \ddots & \ddots & X_1\\
        X_N & \cdots & \cdots & X_1 & X_0
    \end{bmatrix}\quad\text{and}\quad
                \Im Z = \begin{bmatrix} 0 & Y_1 & Y_2 & \cdots & Y_N\\
        -Y_1 & 0 & Y_1 & \ddots & \vdots\\
        -Y_2 & -Y_1 & \ddots & \ddots & \vdots\\
        \vdots & \ddots & \ddots & \ddots & Y_1\\
        -Y_N & \cdots & \cdots & -Y_1 & 0
    \end{bmatrix}\]
    where we have used the assumption that the $X_i$ and the $Y_i$ are symmetric.

    Let $J$ be the $d(N+1)\times d(N+1)$ matrix with $d\times d$ identity
    blocks on the secondary (anti-) block diagonal. Note that left
    multiplication by $J$ reverses the block rows of a block matrix, and right
    multiplication by $J^T = J$ reverses the block columns.  Observe that
    $J(\Re Z)J = \Re Z$ and $J(\Im Z) + (\Im Z)J = 0$.  Let $Q$ denote the
    orthogonal matrix defined by
\[ Q = \frac{1}{\sqrt{2}}\begin{bmatrix} I & -J\\J & I\end{bmatrix}.\]
    A straightforward calculation shows that
\[ QZ_{\R}Q^T = \frac{1}{2}\begin{bmatrix}I & -J\\J & I\end{bmatrix}
\begin{bmatrix} \Re Z & \Im Z\\-\Im Z & \Re Z\end{bmatrix}
\begin{bmatrix} I & J\\-J & I\end{bmatrix} = 
    \begin{bmatrix} \Re Z + J \Im Z & 0\\0 &
    \Re Z + J \Im Z\end{bmatrix}.\]
    So $Z\psd 0$ if and only if $Z_{\R}\psd 0$ which holds if and only if $\Re
    Z + J \Im Z \psd 0$. Finally we note that 
    \[ \Re Z + J\Im Z = \BlkToep(X_0,X_1,\ldots,X_N) + 
    \BlkHank(-Y_N,\ldots,-Y_1,0,Y_1,\ldots,Y_N)\]
    (because reversing the block rows of a block Toeplitz matrix makes it block
    Hankel) to complete the proof.
\end{proof}

\bibliographystyle{plain}
\bibliography{att_bib}

\end{document}